\documentclass[11pt]{amsart}
\usepackage{mathrsfs,amssymb,amsmath,url,xypic}
\theoremstyle{plain}
    \newtheorem{thm}{Theorem}
    
    \newtheorem{prop}[thm]{Proposition}
    \newtheorem{cor}[thm]{Corollary}
    
    \newtheorem{quest}[thm]{Question}

\theoremstyle{definition}

\theoremstyle{remark}

\DeclareMathOperator{\Inv}{Inv} 
\DeclareMathOperator{\Cl}{Cl}\DeclareMathOperator{\Cll}{Cl_{loc}}

\newcommand{\inv}{\Inv}

\newcommand{\um}{^{(m)}}

\newcommand{\cll}[1]{\langle #1 \rangle_{loc}}

\renewcommand{\O}{{\mathscr O}}
\newcommand{\On}{{\mathscr O}^{(n)}}

\newcommand{\Oo}{{\mathscr O}^{(1)}}

\DeclareMathOperator{\ppol}{Pol}

\newcommand{\C}{{\mathscr C}}
\newcommand{\D}{{\mathscr D}}
\newcommand{\F}{{\mathscr F}}

\newcommand{\R}{{\mathscr R}}
\newcommand{\X}{{\mathfrak X}}

\newcommand{\G}{{\mathfrak G}}
\renewcommand{\H}{{\mathfrak H}}

\renewcommand{\S}{{\mathscr S}}

\renewcommand{\L}{{\mathfrak L}}

    \title{Sublattices of the lattice of local clones}

    \author{Michael Pinsker}
    \email{marula@gmx.at}
    \urladdr{http://dmg.tuwien.ac.at/pinsker/}
    \address{Laboratoire de Math\'{e}matiques Nicolas Oresme\\
CNRS UMR 6139\\ Universit\'{e} de Caen\\14032 Caen Cedex\\
France}
    \thanks{The author is grateful for support through project P17812 as well as through the Erwin Schr\"{o}dinger Fellowship of the Austrian Science Fund.}


\begin{document}
    \maketitle
    \begin{abstract}
        We investigate the complexity of the lattice of local clones over
        a countably infinite base set. In particular, we prove
        that this lattice contains all algebraic lattices with at
        most countably many compact elements as complete
        sublattices, but that the class of lattices embeddable into
        the local clone lattice is strictly larger than that.
    \end{abstract}


    \section{Local clones}
        Fix a countably infinite base set $X$, and denote for all $n\geq 1$ the set $X^{X^n}=\{f: X^n\rightarrow X\}$
        of $n$-ary operations on $X$ by $\On$. Then the union
        $\O:=\bigcup_{n\geq 1}\On$ is the set of all finitary
        operations on $X$. A \emph{clone} $\C$ is a subset of $\O$
        satisfying the following two properties:
        \begin{itemize}
            \item $\C$ contains all projections, i.e.\  for all $1\leq
            k\leq n$ the operation $\pi^n_k\in\On$ defined by
            $\pi^n_k(x_1,\ldots,x_n)=x_k$, and

            \item $\C$ is closed under composition, i.e.\  whenever
            $f\in\C$ is $n$-ary and $g_1,\ldots,g_n\in\C$ are
            $m$-ary, then the operation $f(g_1,\ldots,g_n)\in\O\um$ defined by
            $$
                (x_1,\ldots,x_m)\mapsto f(g_1(x_1,\ldots,x_m),\ldots,g_n(x_1,\ldots,x_m))
            $$
            also is an element of $\C$.
        \end{itemize}

        Since arbitrary intersections of clones are again clones, the set of all clones on $X$, equipped with the order of
        inclusion, forms a complete lattice $\Cl(X)$. In this paper, we are not interested
        in all clones of $\Cl(X)$, but only in
        clones which satisfy an additional topological closure property: Equip $X$ with the discrete topology, and
        $\On=X^{X^n}$ with the corresponding product topology (Tychonoff topology), for every
        $n\geq 1$. A clone $\C$ is called \emph{locally closed} or just
        \emph{local} iff each of its $n$-ary fragments $\C\cap\On$ is a closed
        subset of $\On$. Equivalently, a clone $\C$ is local iff it satisfies the
            following interpolation property:\begin{quote}
                For all $n\geq 1$ and all $g\in\On$, if for all finite $A\subseteq X^n$
                there exists an $n$-ary $f\in\C$ which agrees with
                $g$ on $A$, then $g\in\C$.\end{quote}

        Again, taking the set of all local clones on $X$, and
        ordering them according to set-theoretical inclusion, one
        obtains a complete lattice, which we denote by $\Cll(X)$: This is because intersections of
        clones are clones, and because arbitrary intersections of
        closed sets are closed. We are interested in the structure
        of $\Cll(X)$, in particular in how complicated it is as a
        lattice.

        Before we start our investigations, we give an alternative
        description of local clones which will be useful.
        Let $f\in\On$ and let $\rho\subseteq X^m$ be a relation. We
        say that $f$ \emph{preserves} $\rho$ iff
        $f(r_1,\ldots,r_n)\in\rho$ whenever $r_1,\ldots ,r_n\in\rho$,
        where $f(r_1,\ldots,r_n)$ is calculated componentwise.
        For a set of relations $\R$, we write $\ppol(\R)$ for the set
        of those operations in $\O$ which preserve all $\rho\in\R$.
        The operations in $\ppol(\R)$ are called \emph{polymorphisms}
        of $\R$, hence the symbol $\ppol$. The following
        is due to \cite{Rom77}, see also the textbook\ \cite{Sze86}.

        \begin{prop}
            $\ppol(\R)$ is a local clone for all sets of relations
            $\R$. Moreover, every local clone is of this form.
        \end{prop}

        Similarly, for an operation $f\in\On$ and a relation $\rho\subseteq X^m$, we say
        that $\rho$ is \emph{invariant} under $f$ iff $f$ preserves $\rho$.
        Given a set of operations $\F\subseteq\O$, we write
        $\inv(\F)$ for the set of all relations which are invariant
        under all $f\in\F$. Since arbitrary intersections of local clones are local clones again,
        the mapping on the power set of $\O$ which assigns
        to every set of operations $\F\subseteq\O$ the smallest
        local clone $\cll{\F}$ containing $\F$ is a hull operator,
        the closed elements of which are exactly the local clones.
        Using the operators $\ppol$ and $\inv$ which connect
        operations and relations, one obtains the following well-known alternative
        for
        describing this operator (confer \cite{Rom77} or \cite{Sze86}).

        \begin{prop}\label{prop:localClosure}
            Let $\F\subseteq\O$. Then $\cll{\F}=\ppol\Inv(\F)$.
        \end{prop}

        As already mentioned, it is the aim of this paper to investigate the structure of the local clone lattice.
        So far, this lattice has been studied only sporadically, e.g.\ in \cite{RS82-LocallyMaximal}, \cite{RS84-LocalCompletenessI}. There, the
         emphasis was put on finding local completeness criteria for sets of operations $\F\subseteq \O$, i.e.\ on how to decide whether or
        not $\cll{\F}=\O$. Only very recently has the importance of
        the local clone lattice to questions from model theory and
        theoretical computer science been revealed:

        Let $\Gamma=(X,\R)$ be a countably infinite structure; that is, $X$ is a
        countably infinite base set and $\R$ is a set of finitary
        relations on $X$. Consider the expansion $\Gamma'$ of $\Gamma$ by all
        relations which are first-order definable from $\Gamma$. More
        precisely, $\Gamma'$ has $X$ as its base set and its relations
        $\R'$ consist of all finitary relations which can be defined
        from relations in $\R$ using first-order formulas. A
        \emph{reduct} of $\Gamma'$ is a structure $\Delta=(X,\D)$, where $\D\subseteq
        \R'$. We also call $\Delta$ a reduct of $\Gamma$, which essentially
        amounts to saying that we expect our structure $\Gamma$ to be
        closed under first-order definitions. Clearly, the set of
        reducts of $\Gamma$ is in one-to-one correspondence with the
        power set of $\R'$, and therefore not of much interest as a
        partial order. However, it might be more reasonable to consider such
        reducts up to, say, \emph{first-order interdefinability}. That is,
        we may consider two reducts $\Delta_1=(X,\D_1)$ and $\Delta_2=(X,\D_2)$ the same iff their
        first-order expansions coincide, or equivalently iff all
        relations in $\D_1$ are first-order definable in $\Delta_2$ and
        vice-versa.

        In 1976, P.\ J.\ Cameron~\cite{Cameron5} showed that there are exactly five reducts of $(\mathbb Q,<)$ up to
        first-order interdefinability. Recently, M.\ Junker and M.\ Ziegler gave a
        new proof of this fact, and established that $(\mathbb Q,<,a)$, the expansion
        of $(\mathbb Q,<)$ by a constant $a$, has 114 reducts~\cite{JunkerZiegler}.
        S.\ Thomas proved that the first-order theory of the random graph also has exactly five reducts, up to
        first-order interdefinability ~\cite{RandomReducts}.

        These examples have in common that the structures under
        consideration are \emph{$\omega$-categorical}, i.e., their first-order
        theories determine
        their countable models up to isomorphism. This is no
        coincidence: For, given an $\omega$-categorical structure $\Gamma$,
        its reducts up to first-order interdefinability are in
        one-to-one correspondence with the locally closed permutation groups which contain the automorphism
        group of $\Gamma$, providing a tool for describing such reducts (confer \cite{Cam90-Oligo}).

        A natural variant of these concepts
        is to consider reducts up to \emph{primitive positive
        interdefinability}. That is, we consider two reducts
        $\Delta_1, \Delta_2$ of $\Gamma$ the same iff their
        expansions by all relations which are definable from
        each of the structures by primitive positive
        formulas coincide. (A first-order formula is called
        \emph{primitive positive} iff it is of the form $\exists
        \overline x (\phi_1 \wedge \dots \wedge \phi_l)$ for atomic formulas
        $\phi_1,\dots,\phi_l$.)  It turns out that for $\omega$-categorical structures $\Gamma$, the \emph{local
        clones} containing all automorphisms of $\Gamma$ are in one-to-one correspondence with those reducts
        of the first-order expansion of
        $\Gamma$ which are closed under primitive positive
        definitions. This recent connection, which relies on a theorem from \cite{BN06}, has already been utilized in \cite{BCP}, where the reducts of $(\mathbb{N},=)$ have been
        classified by this method (a surprisingly complicated task, as it turned out!).

        We mention in passing that distinguishing relational
        structures up to primitive positive interdefinability, and therefore understanding the structure of $\Cll(X)$, has
        recently
        gained significant importance in theoretical computer
        science, more precisely for what is known as the Constraint
        Satisfaction Problem; see \cite{JBK} or \cite{Bodirsky}.

    \section{The structure of the local clone lattice}

        For our investigations of $\Cll(X)$ we will need the concept
        of an \emph{algebraic lattice}.

        An element $a$ of a complete lattice $\L$ is called \emph{compact}
        iff it has the property that
        whenever $A\subseteq \L$ and $a\leq\bigvee A$, then there exists a finite
        $A'\subseteq A$ with $a\leq \bigvee A'$. $\L$ is called
        \emph{algebraic} iff every element is the supremum of
        compact elements. By their very definition, algebraic
        lattices are determined by their compact elements. More
        precisely, the compact elements form a join-semilattice,
        and every algebraic lattice is isomorphic to the lattice of
        all join-semilattice ideals of the join-semilattice of compact
        elements, see e.g.\  the textbook \cite{CD73}. Whereas the lattice $\Cl(X)$ of all (non-local)
        clones over $X$ is algebraic, it has been discovered
        recently in the survey paper \cite{GP07} that the local clone lattice $\Cll(X)$ is far from being
        so; since that paper is yet to appear, we include a sketch
        of the short proof here.

        \begin{prop}
            The only compact element in the lattice $\Cll(X)$ is
            the clone of projections.
        \end{prop}
        \begin{proof}
            Fix a linear order $\le$ on~$X$ without last element.
            Denote the arity of every $f\in\O$ by $n_f$.
            For each $a\in X$ let
            $$
            \begin{array}{rl}
            \C_a &:= \{f\in \O: \forall  x\in X^{n_f}\,   (f(x)  \le a) \} \quad\text{ and} \\
            \D_a &:= \{f\in \O: \forall x\in X^{n_f}\, ( \max(x)\ge a \
            \Rightarrow  \ f(x) \ge \max (x)) \, \}.
            \end{array}
            $$
            Then
            \begin{enumerate}
            \item $\cll{\C_a} = \C_a \cup \{\,\pi^n_k:1\leq k\leq n<\omega\,\}$.
            \item $\cll{\D_a} $ is the set of all operations which are essentially in
            $\D_a$ (i.e., except for dummy variables).
            \item If $ a \le a'$, then $\C_a \subseteq \C_{a'} $ and $\D_a \subseteq \D_{a'}$, hence every finite union of clones $\cll{\C_a}$ (or~$\cll{\D_a}$, respectively)
            is again a clone of this form.
            \item The local closure of~$\bigcup_a\cll{ \D_a}$, as well as the local
            closure of~$\bigcup_a\cll{ \C_a}$, is the clone of all operations $\O$.
            \item If $f\in \O$ has unbounded range,
            then $f\notin  \bigcup_a \cll{\C_a}$ (unless $f$ is a projection).
            \item If $f\in \O$ has bounded range, then $f\notin  \bigcup_a \cll{\D_a}$.
            \item No local clone  $\C$ (other than the clone of projections)  is
            compact in~$\Cl_{loc}(X)$: If $\C$ contains a nontrivial unbounded
            operation, this is witnessed by the family $(\cll{\C_a}: a\in X)$, and if
            $\C$ contains a bounded operation this is witnessed by the family
            $(\cll{\D_a}: a\in X)$.
            \end{enumerate}
            We leave the easily verifiable details to the reader.
        \end{proof}

        How complicated is $\Cll(X)$, in particular, which lattices
        does it contain as sublattices? The latter question has been posed as ``Problem V'' in the survey paper \cite{GP07}. The following is a first easy
        observation which tells us that there is practically no hope that
        $\Cll(X)$ can ever be fully described, since it is believed
        that already the clone lattice over a three-element set is
        too complex to be fully understood.

        \begin{prop}\label{prop:finiteintolocal}
            Let $\Cl(A)$ be the lattice of all clones over a
            finite set $A$. Then $\Cl(A)$ is an
            isomorphic copy of an interval of $\Cll(X)$.
        \end{prop}

        \begin{proof}
            Assume without loss of generality that $A\subseteq X$.
            Assign to every operation $f(x_1,\ldots,x_n)$ on $A$ a
            set of $n$-ary operations $\S_f\subseteq\On$ on $X$ as
            follows: An operation $g\in \On$ is an element of $\S_f$
            iff $g$ agrees with $f$ on $A^n$. Let $\sigma$ map every
            clone $\C$ on the base set $A$ to the set
            $\bigcup \{\S_f: f\in \C\}$. Then the following hold:
            \begin{enumerate}
                \item
                For every clone $\C$ on $A$, $\sigma(\C)$ is a local
                clone on $X$.
                \item $\sigma$ maps the clone of all operations on
                $A$ to $\ppol(\{A\})$.
                \item All local clones (in fact: all clones) which
                contain $\sigma(\{f: f \text{ is a}\\\text{projection on } A\})$ (i.e., which contain the local clone on $X$ which, via $\sigma$, corresponds to the
                clone of projections on $A$) and
                which are contained in $\ppol(\{A\})$ are of the
                form $\sigma(\C)$ for some clone $\C$ on $A$.
                \item $\sigma$ is one-one and order preserving.
            \end{enumerate}
            (1) and (2) are easy verifications and left to the
            reader. To see (3), let $\D$ be any clone in the
            mentioned interval, and denote by $\C$ the set of all
            restrictions of operations in $\D$ to appropriate powers of
            $A$. Since $\D\subseteq\ppol(\{A\})$, all such
            restrictions are operations on $A$, and since $\D$ is
            closed under composition and contains all projections, so does
            $\C$. Thus, $\C$ is a clone on $A$. We claim
            $\D=\sigma(\C)$. By the definitions of $\C$ and $\sigma$, we have that $\sigma(\C)$ clearly contains
            $\D$. To see the less obvious inclusion, let
            $f\in\sigma(\C)$ be arbitrary, say of arity
            $m$. The restriction of $f$ to $A^m$ is an element of
            $\C$, hence there exists an $m$-ary $f'\in\D$ which has the same
            restriction to $A^m$ as $f$. Define
            $s(x_1,\ldots,x_m,y)\in\O^{(m+1)}$ by
            $$
                s(x_1,\ldots,x_m,y)=\begin{cases} y &,\text{ if}
                (x_1,\ldots,x_m)\in A^m\\ f(x_1,\ldots,x_m)& ,\text{
                otherwise.}\end{cases}
            $$
            Since $s$ behaves on $A^{m+1}$ like the projection
            onto the last coordinate, and since $\D$ contains $\sigma(\{f: f \text{ is } \text{ a } \text{ projection } \text{ on } A\})$, we infer $s\in\D$. But
            $f(x_1,\ldots,x_m)=s(x_1,\ldots,x_m,f'(x_1,\ldots,x_m))$,
            proving $f\in\D$.\\
            (4) is an immediate consequence of (1) and the
            definitions.
        \end{proof}

        It is known that all countable products of finite lattices
        embed into the clone lattice over a four-element set
        \cite{Bul94}, so by the preceding proposition they also
        embed into $\Cll(X)$. However, there are quite simple
        countable lattices which do not embed into the clone lattice
        over any finite set: The lattice $M_\omega$ consisting of a
        countably infinite antichain plus a smallest and a greatest element
        is an example \cite{Bul93}. We shall see now that the class
        of lattices embeddable into $\Cll(X)$ properly contains the
        class of lattices embeddable into the clone lattice over a
        finite set. In fact, the structure of $\Cll(X)$ is at least as
        complicated as the structure of any algebraic lattice with
        $\aleph_0$ compact elements.

        \begin{thm}\label{thm:algebraicIntervals}
            Every algebraic lattice with a countable number of
            compact elements is a complete sublattice of $\Cll(X)$.
        \end{thm}

        To prove Theorem \ref{thm:algebraicIntervals}, we cite
        the following deep theorem from
        \cite{Tum89subgroupLattices}.

        \begin{thm}\label{thm:tuma}
            Every algebraic lattice with a countable number of
            compact elements is isomorphic to an interval in the
            subgroup lattice of a countable group.
        \end{thm}

        \begin{proof}[Proof of Theorem \ref{thm:algebraicIntervals}]
            Let $\L$ be the algebraic lattice to be embedded into
            $\Cll(X)$. Let $\X=(X,+,-,0)$ be the group provided by
            Theorem \ref{thm:tuma}.
            For every $a\in X$, define a unary operation $f_a\in\Oo$ by
            $f_a(x)=a+x$. Clearly, we have
            $f_a(f_b(x))=a+b+x=f_{a+b}(x)$ for all $a,b\in X$. Using
            this, it is easy to verify that
            for all $S\subseteq X$, the (not necessarily local)
            clone $\C_S$
            generated by $\F_S:=\{f_a:a\in S\}$ essentially (that is, up to fictitious variables and projections) consists
            of all operations $f_a$ for which $a$ is in the subsemigroup of $(X,+)$ generated by $S$.
            Let $[\G_1,\G_2]$ be the interval in the
            subgroup lattice of $\X$ that
            $\L$ is isomorphic to. Define a mapping $\sigma:[\G_1,\G_2]\rightarrow \Cll(X)$ sending
            every group
            $\H=(H,+,-,0)$ in the interval to $\C_H$. It follows
            readily from our observation above that the operations in
            $\C_H$ are up to fictitious variables the $f_a$, where $a\in H$, and the projections; in particular, the unary operations in $\C_H$ equal
            $\F_H$ (plus the identity operation, which is an element of $\F_H$ anyway since it equals $f_0$).
            Therefore, $\sigma$ is injective and order-preserving.

            We still have to check that all $\C_H$ are locally closed. To see this, let
            $f\in\cll{\C_H}$; then $f$ depends on only one variable, since all operations in $\C_H$
            depend on only one variable and dependence on several variables is witnessed on finite sets. Assume therefore
            without loss of generality $f\in\Oo$. We claim $f\in \F_H$. To see this, observe
            that $f$ agrees with some $f_a\in \F_H$
            on the finite set
            $\{0\}\subseteq X$. Suppose that there is $b\in X$ such
            that $f(b)\neq f_a(b)=a+b$. Then $f\in\cll{\F_H}$ implies
            that there exists $f_c\in \F_H$ such that $f$ and $f_c$ agree on $\{0,b\}$. But then
            $c=f_c(0)=f(0)=f_a(0)=a$, and thus $f(b)=f_c(b)=c+b=
            a+b=f_a(b)\neq f(b)$, an obvious contradiction. Hence,
            $f=f_a\in \F_H$ and we are done.

            With the explicit description of the $\C_H$ and given that they are indeed local clones, a straightforward check
            shows that $\sigma$ preserves arbitrary meets and joins.
        \end{proof}

        Since in particular, $\Cll(X)$ contains $M_\omega$ as a
        sublattice, and since according to \cite{Bul93}, $M_\omega$ is not a sublattice of the
        clone lattice over any finite set, we have the following
        corollary to Theorem \ref{thm:algebraicIntervals}.

        \begin{cor}\label{cor:localnotinfinite}
            $\Cll(X)$ does not embed into the clone lattice over any
            finite set.
        \end{cor}

        Observe also that Theorem \ref{thm:algebraicIntervals} is a
        strengthening of Proposition \ref{prop:finiteintolocal} in
        so far as the clone lattice over a finite set is an example
        of an algebraic lattice with countably many compact
        elements. However, in that proposition we obtain an
        embedding as an interval, not just as a complete sublattice.

        What about other lattices, i.e.\ lattices which are more
        complicated or larger than algebraic lattices with countably
        many compact elements? The following proposition puts a restriction
        on which lattices can be sublattices of $\Cll(X)$.

        \begin{prop}\label{prop:intopowersetofomega}
            $\Cll(X)$ embeds as a suborder into the power set of $\omega$. In
            particular, it does not contain any uncountable
            ascending or descending chains.
        \end{prop}

        \begin{cor}\label{cor:cizeOfTheCloneLattice}
            The size of $\Cll(X)$ is $2^{\aleph_0}$.
        \end{cor}

        \begin{proof}[Proof of Corollary
        \ref{cor:cizeOfTheCloneLattice}]
            The fact that all algebraic lattices with at most
            $\aleph_0$ compact elements embed into $\Cll(X)$ shows
            that it must contain at least $2^{\aleph_0}$ elements (since for example the power set of $\omega$ with
            inclusion is such an algebraic lattice).
            The upper bound is a consequence of Proposition
            \ref{prop:intopowersetofomega}.
        \end{proof}

        In order to see the truth of Proposition
        \ref{prop:intopowersetofomega}, the following definition will be
        convenient.

        A \emph{partial clone of finite operations} on $X$ is a set of partial
        operations of finite domain
        on $X$ which contains all restrictions of the projections to finite domains and which is closed
        under composition. The set of partial clones of finite operations on $X$
        forms a complete algebraic lattice, the compact elements of
        which are precisely the finitely generated partial clones.

        \begin{prop}
            The mapping $\sigma$ from $\Cll(X)$ into the lattice of partial clones of finite operations on $X$ which sends every $\C\in\Cll(X)$
            to the partial clone of all restrictions of its operations
            to finite
            domains is one-to-one and preserves arbitrary joins.
        \end{prop}
        \begin{proof}
            It is obvious that $\sigma(\C)$ is a partial clone of
            finite operations, for all local (in fact: also non-local) clones $\C$.\\
            Let $\C,\D\in\Cll(X)$ be distinct. Say without loss of generality that there is an
            $n$-ary
            $f\in\C\setminus\D$; then since $\D$ is locally closed, there exists some finite set
            $A\subseteq X^n$ such that there is no $g\in \D$ which
            agrees with $f$ on $A$. The restriction of $f$ to $A$
            then witnesses that $\sigma(\C)\neq\sigma(\D)$.\\
            We show that
            $\sigma(\C)\vee\sigma(\D)=\sigma(\C\vee\D)$; the proof for arbitrary
            joins works the same way. It follows directly from the
            definition of $\sigma$ that it is
            order-preserving. Thus, $\sigma(\C\vee\D)$ contains
            both $\sigma(\C)$ and $\sigma(\D)$ and hence also their
            join. Now let $f\in \sigma(\C)\vee\sigma(\D)$. This
            means that it is a composition of partial operations in
            $\sigma(\C)\cup\sigma(\D)$. All partial operations used
            in this composition have extensions to operations in
            $\C$ or $\D$, and if we compose these extensions in the
            same way as the partial operations, we obtain an
            operation in $\C\vee\D$ which agrees with $f$ on the
            domain of the latter. Whence, $f\in\sigma(\C\vee\D)$.
        \end{proof}

        Note that the preceding proposition immediately implies
        Proposition \ref{prop:intopowersetofomega}: The number of partial operations with finite
        domain on $X$ is countable, and therefore partial clones of finite operations can be considered as
        subsets of $\omega$.

        Until today, no other restriction to embeddings into
        $\Cll(X)$ except for Proposition
        \ref{prop:intopowersetofomega} is known, and we ask:

        \begin{quest}
            Does every lattice which is order embeddable into the power set of $\omega$ have a lattice embedding into $\Cll(X)$?
        \end{quest}

        However, it seems difficult to embed even the simplest
        lattices which are not covered by Theorem
        \ref{thm:algebraicIntervals} into $\Cll(X)$. For example, we
        do not know:

        \begin{quest}
            Does the lattice $M_{2^{\aleph_0}}$, which consists of an antichain of length $2^{\aleph_0}$ plus a smallest and a largest element, embed
            into $\Cll(X)$?
        \end{quest}

        So far, we only know

        \begin{prop}
            There exists a join-preserving embedding as well as a meet-preserving embedding of
            $M_{2^{\aleph_0}}$ into $\Cll(X)$.
        \end{prop}
        \begin{proof}
            Denote by $0$ and $1$ the smallest and the largest element of $M_{2^{\aleph_0}}$, respectively, and enumerate the elements of its antichain by
            $(a_i)_{i\in {2^{\aleph_0}}}$.\\ We first construct a join-preserving embedding.
            Enumerate the non-empty proper subsets of $X$ by $(A_i)_{i\in 2^{\aleph_0}}$. Consider the mapping $\sigma$ which sends
            $0$ to the clone of projections, $1$ to $\O$, and every
            $a_i$ to $\ppol(\{A_i\})$. Now it is well-known (see \cite{RS84-LocalCompletenessI}) that for any non-empty proper subset $A$ of $X$,
            $\ppol(\{A\})$ is
            covered by $\O$, i.e.\ there exist no local (in fact even no global) clones between $\ppol(\{A\})$ and $\O$. Hence, we have that
            $\sigma(a_i)\vee\sigma(a_j)=\cll{\ppol(A_i)\cup\ppol(A_j)}=\O=\sigma(1)$ for all $i\neq j$. Since clearly $\sigma(a_i)$ contains $\sigma(0)$ for all $i\in 2^{\aleph_0}$, the mapping $\sigma$
            indeed preserves joins.\\
            To construct a meet embedding, fix any distinct $a,b\in X$ and
            define for every non-empty subset $A$ of $X\setminus\{a,b\}$ an
            operation $f_A\in\Oo$ by
            $$
                f_A(x)=\begin{cases}a, &\text{if } x\in A\\b,
                &\text{otherwise.}
                \end{cases}
            $$
            Enumerate the non-empty subsets of $X\setminus\{a,b\}$ by $(B_i:i\in 2^{\aleph_0})$.
            Denote the constant unary operation with value $b$ by $c_b$. Let the embedding $\sigma$ map $0$ to
            $\cll{\{c_b\}}$, for all $i\in 2^{\aleph_0}$ map $a_i$ to $\cll{\{f_{B_i}\}}$, and
            let it map $1$ to $\O$. One readily checks that
            $\sigma(a_i)=\cll{\{f_{B_i}\}}$ contains only
            projections and, up to fictitious variables, the operations $f_{B_i}$ and $c_b$.
            Therefore, for $i\neq j$ we have $\sigma(a_i)\wedge
            \sigma(a_j)=\cll{\{c_b\}}=\sigma(0)$. Since clearly
            $\sigma(a_i)\subseteq \sigma(1)=\O$ for all $i\in
            2^{\aleph_0}$, we conclude that $\sigma$ does indeed
            preserve meets.

        \end{proof}

        Simple as the preceding proposition is, it still shows us as a
        consequence that Theorem \ref{thm:algebraicIntervals} is not
        optimal.

        \begin{cor}
            $\Cll(X)$ is not embeddable into any algebraic lattice
            with countably many compact elements.
        \end{cor}
        \begin{proof}
            It is well-known and easy to check (confer also \cite{CD73}) that any algebraic lattice $\L$ with countably many compact elements can be
            represented as the subalgebra lattice of an algebra over
            the base set
            $\omega$. The meet in the subalgebra lattice $\L$ is just
            the set-theoretical intersection. Now there is certainly
            no uncountable family of subsets of $\omega$ with the
            property that any two distinct members of this family have
            the same intersection $D$; for the union of such a family would have to be uncountable. Consequently, $\L$ cannot
            have $M_{2^{\aleph_0}}$ as a meet-subsemilattice. But $\Cll(X)$ has, hence $\L$ cannot have $\Cll(X)$ as a sublattice.
        \end{proof}

        Observe that this corollary is a strengthening of Corollary
        \ref{cor:localnotinfinite}, since the clone lattice over a
        finite set is an algebraic lattice with countably many
        compact elements.

        We conclude by remarking that the lattice $\Cl(X)$ of all (not
        necessarily local) clones on $X$ is infinitely more
        complicated than $\Cll(X)$: It contains all algebraic
        lattices with at most $2^{\aleph_0}$ compact elements, and
        in particular all lattices of size continuum, as complete
        sublattices \cite{Pin06AlgebraicSublattices}.

\def\ocirc#1{\ifmmode\setbox0=\hbox{$#1$}\dimen0=\ht0 \advance\dimen0
  by1pt\rlap{\hbox to\wd0{\hss\raise\dimen0
  \hbox{\hskip.2em$\scriptscriptstyle\circ$}\hss}}#1\else {\accent"17 #1}\fi}

\end{document}